\documentclass[table]{amsart}
\usepackage[margin=1.1in]{geometry}
\usepackage{amscd,amsmath,amsxtra,amsthm,amssymb,stmaryrd,xr,mathrsfs,mathtools,enumerate,commath, comment, mathtools, orcidlink}
\usepackage{tikz}
\usetikzlibrary{arrows.meta,positioning}

\usetikzlibrary{calc}
\usepackage{tikz-3dplot}
\usepackage{stmaryrd}
\usepackage{multirow}
\usepackage{xcolor}
\usepackage{commath}
\usepackage{comment}
\usepackage{svg}
\usepackage{graphics}
\usepackage{longtable} 
\usepackage{pdflscape} 
\usepackage{booktabs}
\usepackage{hyperref}
\definecolor{vegasgold}{rgb}{0.77, 0.7, 0.35}
\definecolor{darkgoldenrod}{rgb}{0.72, 0.53, 0.04}
\definecolor{gold(metallic)}{rgb}{0.83, 0.69, 0.22}
\hypersetup{
 colorlinks=true,
 linkcolor=darkgoldenrod,
 filecolor=brown,      
 urlcolor=gold(metallic),
 citecolor=darkgoldenrod,
 }
\newtheorem{lthm}{Theorem}

\usepackage[all,cmtip]{xy}

\DeclareFontFamily{U}{wncy}{}
\DeclareFontShape{U}{wncy}{m}{n}{<->wncyr10}{}
\DeclareSymbolFont{mcy}{U}{wncy}{m}{n}
\DeclareMathSymbol{\Sh}{\mathord}{mcy}{"58}
\usepackage[T2A,T1]{fontenc}
\usepackage[OT2,T1]{fontenc}

\newtheorem{theorem}{Theorem}[section]
\newtheorem{lemma}[theorem]{Lemma}

\newtheorem*{theorem*}{Theorem}
\newtheorem*{ass*}{Assumption}
\newtheorem{definition}[theorem]{Definition}
\newtheorem{corollary}[theorem]{Corollary}

\newtheorem{proposition}[theorem]{Proposition}

\newcommand{\tr}{\operatorname{tr}}

\newcommand{\Z}{\mathbb{Z}}
\newcommand{\Q}{\mathbb{Q}}
\newcommand{\F}{\mathbb{F}}

\newcommand{\cC}{\mathcal{C}}
\newcommand{\cO}{\mathcal{O}}

\newcommand{\op}[1]{\operatorname{#1}}
\newcommand\mtx[4] { \left( {\begin{array}{cc}
 #1 & #2 \\
 #3 & #4 \\
 \end{array} } \right)}

 \DeclareMathSymbol{\sha}{\mathord}{mcy}{"58}
 \makeatletter
\newcommand{\mylabel}[2]{#2\def\@currentlabel{#2}\label{#1}}
\makeatother

\numberwithin{equation}{section}

\title{Distribution questions for Isogeny Graphs over Finite Fields}
\author[A.~Ray]{Anwesh Ray\, \orcidlink{0000-0001-6946-1559}}
\address[Ray]{Chennai Mathematical Institute, H1, SIPCOT IT Park, Kelambakkam, Siruseri, Tamil Nadu 603103, India}
\email{anwesh@cmi.ac.in}

\begin{document}

\maketitle

\begin{abstract}
In the first part of the paper, we fix a non-CM elliptic curve $E/\Q$ and an odd prime $\ell$ and investigate the distribution of invariants associated to the $\ell$–volcano containing the reduction $E_p$, as $p$ ranges over primes of good ordinary reduction. Let $H(p)$ be the height of the volcano and $d'(p)$ denote the relative position of $j(E_p)$ above the floor and let $r\geq 0$ be an integer. Assuming that the $\ell$–adic Galois representation attached to $E$ is surjective, we derive an explicit formula for the natural density of primes $p$ for which $H(p)=r$ (resp. $d'(p)=r$). In the non-surjective case, we show that all sufficiently large heights occur with positive density. In the second part of the paper, we analyze the distribution of $\ell$–volcano heights over a finite field $\F_q$ and consider the limit as $q\rightarrow \infty$. Using analytic estimates for sums of Hurwitz class numbers in arithmetic progressions, we compute exact limiting densities for ordinary elliptic curves whose $\ell$–isogeny graph has a prescribed height $r$.
\end{abstract}

\section{Introduction}

\subsection{Motivation}

Let $k=\F_q$ be a finite field of characteristic $p$, and let $\ell$ be a prime number (not necessarily distinct from $p$ unless otherwise specified). The $\ell$–isogeny multigraph $\mathcal{G}_\ell(k)$ is defined as follows. Its vertex set consists of the $j$–invariants of elliptic curves defined over $\F_q$, taken up to $\F_q$–isomorphism. Given two vertices $j(E_1)$ and $j(E_2)$, we draw an edge between them for each isogeny $\varphi\colon E_1\to E_2$ of degree $\ell$. In particular, multiple edges may occur, reflecting the fact that there may exist several distinct $\ell$–isogenies between a fixed pair of isomorphism classes. Every isogeny $\varphi\colon E_1\to E_2$ admits a dual isogeny $\widehat{\varphi}\colon E_2\to E_1$ of the same degree, and the compositions $\widehat{\varphi}\circ\varphi$ and $\varphi\circ\widehat{\varphi}$ are equal to multiplication by $\ell$ on $E_1$ and $E_2$, respectively. As a consequence, the multigraph $\mathcal{G}_\ell(k)$ is naturally undirected. The structure of $\mathcal{G}_\ell(k)$ reflects arithmetic properties of elliptic curves over finite fields, including the behavior of endomorphism rings, the splitting of primes in imaginary quadratic fields. From a graph-theoretic perspective, the connected components of $\mathcal{G}_\ell(k)$ exhibit striking and rigid patterns. In the ordinary case, these components are \emph{$\ell$–volcanoes}, while in the supersingular case they form highly connected Ramanujan graphs. The systematic study of these graphs and their structural properties was initiated by Kohel \cite{kohel}.

\par Beyond their intrinsic arithmetic interest, $\ell$-isogeny graphs have attracted significant attention in recent years due to their role in post-quantum cryptography. Cryptographic protocols based on the presumed hardness of finding isogenies between elliptic curves exploit the combinatorial complexity and expansion properties of these graphs, particularly in the supersingular setting. Charles, Goren and Lauter \cite{CGL} introduced a Hash function associated to an expander graph, and explored cryptographic applications. Some examples of cryptosystems include the SQISign cryptosystem \cite{crypto1} and SCALLOP \cite{crypto2}. This has further motivated a detailed investigation of the combinatorial structure of supersingular isogeny graphs, see for instance \cite{colokohel, arpinetal, orvis, cyclesandcuts}.
\par When restricted to ordinary elliptic curves, the stratified nature of volcanoes reflects the variation of endomorphism rings under $\ell$–isogenies and is governed by the $\ell$–adic valuation of discriminants of associated quadratic orders. The height of a volcano measures how far the order generated by Frobenius is from being maximal, while the depth of a given vertex records the valuation of the conductor of the corresponding endomorphism ring. A natural problem is to understand how these invariants vary statistically, either when one fixes a finite field and ranges over elliptic curves, or when one fixes an elliptic curve over $\Q$ and studies its reductions modulo primes. Such questions link the geometry of isogeny graphs to classical problems in analytic number theory, including the distribution of traces of Frobenius and the arithmetic of imaginary quadratic orders. In the finite field setting, counts of elliptic curves with prescribed volcano height are governed by Hurwitz class numbers, while in the global setting they may be interpreted through Galois representations and the Chebotarev density theorem.

\subsection{Main results}

In the first part of the paper, we fix a non-CM elliptic curve $E/\Q$ and an odd prime $\ell$. By a well known result of Serre, the set of primes $p$ at which $E$ has good ordinary reduction has density $1$. We study the variation of $\ell$–volcanoes associated to the reductions $E_p$ as $p$ ranges over primes of good ordinary reduction. Writing $H(p)$ for the height of the $\ell$–volcano containing $j(E_p)$ and $d(p)$ for the depth of $j(E_p)$ within that volcano, we express these quantities in terms of the $\ell$–adic valuation of the discriminant $a_p^2-4p$. We also set $d'(p):=H(p)-d(p)$. Interpreting the conditions $H(p)=r$ and related constraints on $d(p)$ as congruence conditions on Frobenius elements in $\op{GL}_2(\Z_\ell)$, we apply Chebotarev’s density theorem to obtain precise densities for primes with prescribed volcano invariants. In particular, assuming surjectivity of the $\ell$–adic Galois representation, we obtain exact formulas for the density of primes $p$ with $H(p)=r$, as well as for finer invariants measuring the distance from the floor of the volcano.

\begin{lthm}[Theorem \ref{section 3 thm 1}]\label{thma}
     Let $\ell$ be an odd prime such that $\ell\neq \op{char}k$ and assume that $E$ is a non-CM elliptic curve for which the $\ell$-adic Galois representation is surjective. Then for $r>0$, the density of primes $p$ at which $E$ has good ordinary reduction and $H(p)=r$ is precisely 
  \[\begin{split}&2\ell^{-2r}(1-\ell^{-2})
+
\frac{\ell^{-2r}}{\ell^{2}-1}
\Big(
2+2\ell-2\ell^{-2}-2\ell^{-1}
-(2r+2)\ell^{1-r}
+(2r+4)\ell^{-r-2}
\Big)\\
=& 2\ell^{-2r}+O_r(\ell^{-2r-1}),\end{split}\]where the implied constant in $O_r$ depends on $r$ but not on $\ell$. On the other hand, if $r=0$, the density equals 
\[
1-2\ell^{-2}-\frac{2\ell^{-2}}{\ell+1}.
\]
\end{lthm}

\begin{lthm}[Theorem \ref{section 3 thm 2}]
    Let $E$ be as in Theorem \ref{thma}. Then for $r>0$, the density of primes $p$ at which $E$ has good ordinary reduction and $d'(p)=r$ is precisely $\frac{\ell^2+\ell+1}{(\ell+1)\ell^{3r+1}}$. The density of primes $p$ for which $d'(p)=0$ equals $\frac{\ell^3-\ell-1}{\ell(\ell-1)(\ell+1)}$.
\end{lthm}

Next consider a non-CM elliptic curve $E_{/\Q}$ for which the $\ell$-adic Galois representation is not surjective. Then by Serre's open image theorem, there is a minimal integer $k\geq 1$ such that its image contains $\mathcal{G}^k:=\op{ker}\left(\op{GL}_2(\Z_\ell)\longrightarrow \op{GL}_2(\Z/\ell^k)\right)$. 
\begin{lthm}[Theorem \ref{section 3 thm 3}]
    Let $\ell\neq \op{char}k$ be an odd prime number and assume that $E$ is a non-CM elliptic curve over $\Q$. Let $k\geq 1$ be as above and $r\geq k$. Then, the density of primes $p\neq \ell$ for which $E$ has good ordinary reduction and $H(p)=d'(p)=r$ is positive. 
\end{lthm}

The second part of this paper studies the distribution of ordinary elliptic curves over finite fields according to the height of their associated $\ell$–volcano. Let $\mathcal{E}(\F_q)$ be the set of isomorphism classes of elliptic curves over $\F_q$. Fixing an odd prime $\ell\neq p$ and an integer $r\ge 0$, we consider the subset $\mathcal{E}(r;\F_q)$ of ordinary elliptic curves over $k$ whose $\ell$–isogeny graph component has height exactly $r$. Using Deuring’s correspondence between isogeny classes and imaginary quadratic orders, together with estimates for sums of Hurwitz class numbers in arithmetic progressions, we compute the limiting density
\[
\mathfrak{d}_r=\lim_{q\to\infty}\frac{\#\mathcal{E}(r;\F_q)}{\#\mathcal{E}(\F_q)}.
\]
\begin{lthm}[Theorem \ref{thm:density}]
We have that
\[
\mathfrak{d}_r
=
\begin{cases}\frac{\ell^{2r}(\ell^2-1)(\ell^{4r+2}+1)}
{(\ell^{4r}-1)(\ell^{4r+4}-1)}&\text{ if }r\geq 1,\\
\left(1-\frac{\ell^2}{\ell^4-1}\right)&\text{ if }r=0.
\end{cases}
\]
\end{lthm}

\subsection{Organization of the paper}

Including the introduction, the article consists of $4$ sections. Section~\ref{s2} recalls the basic structure of $\ell$–isogeny graphs over finite fields, including the classification of ordinary components as volcanoes and their description in terms of endomorphism rings. In Section~\ref{s3}, we turn to elliptic curves over $\Q$ and analyze the variation of volcano height in families of reductions. After reviewing the relevant properties of adelic and $\ell$–adic Galois representations, we translate conditions on volcano invariants into conditions on Frobenius conjugacy classes. Section~\ref{s4} studies elliptic curves over finite fields and establishes asymptotic formulas for the number of curves whose $\ell$–isogeny graph has a given height.

\subsection{Outlook}
The themes studied in this article could lead to further developments in the supersingular setting as well, where quaternionic methods replace complex multiplication. It would also be interesting to explore extensions to higher-dimensional abelian varieties, cf. \cite{abvar1, abvar2}. Given the growing importance of isogeny graphs in cryptographic applications, a deeper understanding of their statistical properties may have further implications in cryptography.

\subsection*{Data availability} No data was analyzed in proving the results in the article.

\subsection*{Conflict of interest} There is no conflict of interest to report.

\section{Isogeny graphs of elliptic curves}\label{s2}

\par
In this section, we recall the basic structure of isogeny graphs associated to the set of elliptic curves defined over a finite field $k=\F_q$ of characteristic $p$. The absolute Galois group $\op{G}_k=\op{Gal}(\bar{k}/k)$ is topologically generated by the Frobenius automorphism $\phi:x\mapsto x^q$. Let $E/k$ be an elliptic curve with
\[
j(E)=j(a,b)=1728\,\frac{4a^3}{4a^3+27b^2}.
\]
The curve $E$ is said to be \emph{supersingular} if $E(\bar{k})[p]=0$, and \emph{ordinary} otherwise. In the ordinary case, $\op{End}(E)$ is an order in an imaginary quadratic field, whereas in the supersingular case it is a maximal order in a quaternion algebra.

\par
An \emph{isogeny} $\varphi:E_1\to E_2$ is a nonzero morphism of elliptic curves sending $0$ to $0$. It is finite of degree $\deg(\varphi)$ and admits a dual isogeny $\widehat{\varphi}:E_2\to E_1$ satisfying $\widehat{\varphi}\circ\varphi=[n]_{E_1}$ when $\deg(\varphi)=n$. If $\gcd(n,p)=1$, then $\varphi$ is separable with $|\ker(\varphi)|=\deg(\varphi)$. For any prime $\ell\ne p$, one has $E[\ell]\simeq(\Z/\ell\Z)^2$, which contains exactly $\ell+1$ cyclic subgroups of order $\ell$, each corresponding to a separable $\ell$–isogeny. Such an isogeny is defined over $k$ precisely when its kernel is stable under the action of $\op{G}_k$ on $E[\ell]$. This action gives rise to the mod-$\ell$ Galois representation
\[
\rho_{E,\ell}: \op{G}_k\longrightarrow\op{Aut}(E[\ell])\cong\op{GL}_2(\F_\ell).
\]

\begin{lemma}
Let $E/k$ be an elliptic curve with $j(E)\notin\{0,1728\}$, and let $\ell\ne p$ be a prime. Then the number of $k$–rational $\ell$–isogenies with source $E$ is $0$, $1$, $2$, or $(\ell+1)$.
\end{lemma}

\begin{proof}
Let $G$ denote the image of the projective representation
\[
\op{G}_k \longrightarrow \op{PGL}_2(\F_\ell).
\]
Since $j(E)\notin\{0,1728\}$, the automorphism group $\op{Aut}(E)$ is $\{\pm 1\}$, and therefore distinct $\ell$–isogenies correspond bijectively to $G$–stable lines in the two–dimensional $\F_\ell$–vector space $E[\ell]$. If $G$ acts trivially, all $\ell+1$ lines in $\mathbb{P}(E[\ell])\simeq \mathbb{P}^1(\F_\ell)$ are fixed. Otherwise, any nontrivial element of $\op{PGL}_2(\F_\ell)$ fixes at most two points of $\mathbb{P}^1(\F_\ell)$, since an element fixing three points must be the identity, and the result follows.
\end{proof}

\par
Since $\op{G}_k$ is procyclic, generated by the Frobenius $\pi(x)=x^q$, the image $\sigma:=\rho_{E,\ell}(\pi)\in\op{GL}_2(\F_\ell)$ determines the $G$–action on $E[\ell]$. A line is fixed by $G$ if and only if it is an eigenspace of $\sigma$. By the Weil pairing, $\det(\rho_{E,\ell})$ is the mod-$\ell$ cyclotomic character, so
\[
\det(\sigma)\equiv q \pmod{\ell}.
\]
The characteristic polynomial of $\sigma$ is
\[
x^2 - t x + q,
\qquad\text{where } t=\operatorname{trace}(\sigma).
\]
By the Hasse–Weil bound, $t\in[-2\sqrt{q},\,2\sqrt{q}]$. Let
\[
\Delta_\pi := t^2-4q \le 0
\]
be the discriminant. According to the structure of the Frobenius action, one obtains exactly the following possibilities:

\begin{itemize}
    \item $\sigma$ is a scalar matrix; in this case all $(\ell+1)$ lines are fixed, giving $(\ell+1)$ $k$–rational $\ell$–isogenies.
    \item $\sigma$ is diagonalizable over $\F_\ell$ with two distinct eigenvalues; in this case, exactly two lines are fixed, giving two $k$–rational $\ell$–isogenies.
    \item $\sigma$ has a single eigenvalue in $\F_\ell$ but is non-diagonalizable; it is then conjugate to a matrix of the form $\begin{psmallmatrix}\lambda & * \\ 0 & \lambda\end{psmallmatrix}$, and exactly one line is fixed.
    \item $\sigma$ has no eigenvalues in $\F_\ell$; in this case no line is fixed, and there are no $k$–rational $\ell$–isogenies.
\end{itemize}
\par Before discussing isogeny graphs, let us briefly recall the notion of a \emph{multigraph}. A multigraph $\Gamma$ consists of a finite set $V$ of \emph{vertices}, a set $E^+$ of \emph{directed edges} and an adjacency map $\alpha: E^+\rightarrow V\times V$ which maps $e\in E^+$ to a pair $(v, v')$ where $e$ is an edge from $v$ to $v'$. The multigraph is said to be \emph{undirected} if for each $e\in E^+$, there is an inverse $\bar{e}\in E^+$ such that $\alpha(\bar{e})=(v',v)$. In other words, inversion $\iota(e):=\bar{e}$ is a bijection $\iota: E^+\rightarrow E^+$ such that:
\begin{itemize}
    \item $\tau^2=\op{Id}$,
    \item $\alpha\circ \iota=\mu\circ \alpha$, where $\mu(v,v'):=(v',v)$.
\end{itemize} For an undirected multigraph, set $E:=E^+/\sim$ where $e\sim e'$ if $e'=e$ or $e'=\bar{e}$. The set $E$ are the edges of $\Gamma$. Note that an undirected multigraph allows for self-loops and multiple edges between two vertices. For the rest of this article, by a \emph{graph}, we simply mean an undirected multigraph. Two vertices $v,w \in V$ are said to be \emph{adjacent} (or \emph{neighbours}) if $\{v,w\}\in E$. The \emph{degree} of a vertex $v$, denoted $\deg(v)$, is the number of edges incident to $v$; equivalently, it is the number of vertices adjacent to $v$. A graph is \emph{$k$–regular} if every vertex has degree $k$. A subset $W \subseteq V$ determines an \emph{induced subgraph} of $\Gamma$, consisting of the vertices in $W$ and all edges between them. A graph is \emph{connected} if every pair of vertices lies in a common path, that is, a sequence of edges joining them.

\par
We consider the graph $\mathcal{G}_\ell(k)$ whose vertices are the $j$-invariants of elliptic curves defined over $k$. Given $j_1=j(E_1)$ and $j_2=j(E_2)$, we draw an edge between $j_1$ and $j_2$ if and only if there exists an $\ell$-cyclic isogeny $E_1\!\to\!E_2$. Since every isogeny admits a dual isogeny, the relation is symmetric, and therefore $\mathcal{G}_\ell(k)$ is an undirected multigraph. There is possibly more than one undirected edge between any two vertices, as well as self-loops.
\par The multigraph $\mathcal{G}_\ell(k)$ decomposes as a disjoint union of its connected components. If $E_1$ is ordinary (resp.\ supersingular) and there exists an isogeny $E_1\to E_2$, then $E_2$ is also ordinary (resp.\ supersingular). Hence each connected component of $\mathcal{G}_\ell(k)$ consists entirely of ordinary elliptic curves or entirely of supersingular elliptic curves. Accordingly, we refer to a connected component of $\mathcal{G}_\ell(k)$ as \emph{ordinary} or \emph{supersingular} depending on the nature of the curves it contains. Ordinary components of $\mathcal{G}_\ell(k)$ are instances of graphs known as \emph{volcanoes}. Their structure reflects the way endomorphism rings of elliptic curves change under $\ell$-isogenies. The component is stratified into levels, where the level of a vertex records the index of the endomorphism ring of the corresponding curve inside the endomorphism ring of an elliptic curve at the \emph{crater} of the volcano.

\begin{definition}
Let $\ell$ be a prime number. An \emph{$\ell$-volcano} $\mathcal{G}=(V,E)$ of height $H$ is a connected undirected graph whose vertex set $V$ is partitioned into disjoint levels
\[
V \;=\; \bigsqcup_{i=0}^H V_i,
\]
satisfying the following properties:
\begin{enumerate}
    \item The subgraph induced by $V_0$ (the \emph{crater}) is regular of degree at most $2$. Thus the crater is either a cycle, a pair of vertices joined by a double edge, a single edge, or a single vertex.
    \item For each $i>0$ and every vertex $v\in V_i$, there is a unique neighbour of $v$ lying in $V_{i-1}$.
    \item For each $i<H$, every vertex in $V_i$ has total degree $(\ell+1)$.
\end{enumerate}
We refer to $V_0$ as the \emph{crater}, to $V_H$ as the \emph{floor}. The quantity $H=H(\mathcal{G})$ is a measure of the overall complexity of the volcano. 
\end{definition}
\noindent For instance, here's a $2$-volcano with $H=3$:
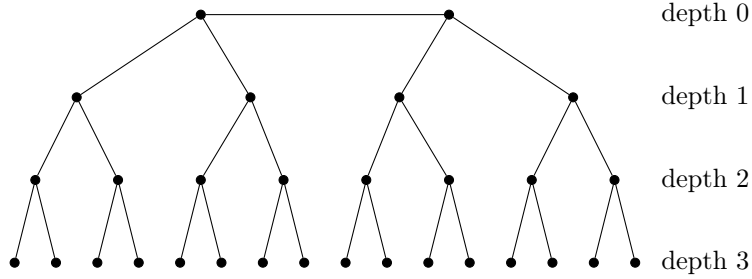
\begin{center}
\begin{tikzpicture}[
  vertex/.style={circle, draw, fill=black, inner sep=1.2pt},
  scale=1.1
]

\node[vertex] (c1) at (-1,3) {};
\node[vertex] (c2) at (2,3) {};
\draw (c1) -- (c2);

\node[vertex] (d11) at (-2.5,2) {};
\node[vertex] (d12) at (-0.4,2) {};
\node[vertex] (d13) at (1.4,2) {};
\node[vertex] (d14) at (3.5,2) {};

\draw (c1) -- (d11);
\draw (c1) -- (d12);
\draw (c2) -- (d13);
\draw (c2) -- (d14);

\node[vertex] (d21) at (-3,1) {};
\node[vertex] (d22) at (-2,1) {};
\node[vertex] (d23) at (-1,1) {};
\node[vertex] (d24) at (0,1) {};
\node[vertex] (d25) at (1,1) {};
\node[vertex] (d26) at (2,1) {};
\node[vertex] (d27) at (3,1) {};
\node[vertex] (d28) at (4,1) {};

\draw (d11) -- (d21);
\draw (d11) -- (d22);
\draw (d12) -- (d23);
\draw (d12) -- (d24);
\draw (d13) -- (d25);
\draw (d13) -- (d26);
\draw (d14) -- (d27);
\draw (d14) -- (d28);

\node[vertex] (d31) at (-3.25,0) {};
\node[vertex] (d32) at (-2.75,0) {};
\node[vertex] (d33) at (-2.25,0) {};
\node[vertex] (d34) at (-1.75,0) {};
\node[vertex] (d35) at (-1.25,0) {};
\node[vertex] (d36) at (-0.75,0) {};
\node[vertex] (d37) at (-0.25,0) {};
\node[vertex] (d38) at (0.25,0) {};
\node[vertex] (d39) at (0.75,0) {};
\node[vertex] (d310) at (1.25,0) {};
\node[vertex] (d311) at (1.75,0) {};
\node[vertex] (d312) at (2.25,0) {};
\node[vertex] (d313) at (2.75,0) {};
\node[vertex] (d314) at (3.25,0) {};
\node[vertex] (d315) at (3.75,0) {};
\node[vertex] (d316) at (4.25,0) {};

\draw (d21) -- (d31);
\draw (d21) -- (d32);
\draw (d22) -- (d33);
\draw (d22) -- (d34);
\draw (d23) -- (d35);
\draw (d23) -- (d36);
\draw (d24) -- (d37);
\draw (d24) -- (d38);
\draw (d25) -- (d39);
\draw (d25) -- (d310);
\draw (d26) -- (d311);
\draw (d26) -- (d312);
\draw (d27) -- (d313);
\draw (d27) -- (d314);
\draw (d28) -- (d315);
\draw (d28) -- (d316);

\node[draw=none] at (5.1,3) {depth $0$};
\node[draw=none] at (5.1,2) {depth $1$};
\node[draw=none] at (5.1,1) {depth $2$};
\node[draw=none] at (5.1,0) {depth $3$};

\end{tikzpicture}
\end{center}

\noindent It is possible for there to be self-loops along the crater, but nowhere else in the volcano.
\par Let $E_{/k}$ be an ordinary elliptic curve and $\ell\neq p$ be a prime. The endomorphism ring $\cO_E:=\op{End}(E)$ is an order in the imaginary quadratic field $K:=\Q(\sqrt{\Delta_\pi})$. Such an order is contained in the ring of integers $\cO_K$ and is in particular, of the form
\[\cO_E=\Z+f_E\cO_K\] where $f_E\geq 1$ is an integer known as the \emph{conductor} of $\cO_E$. One has the inclusions 
\[\Z[\pi]\subseteq \cO_E\subseteq \cO_K. \]

\begin{proposition}[Kohel]
    Let $E$ and $E'$ be elliptic curves over $k$ and suppose that there is an $\ell$-isogeny $\varphi: E\rightarrow E'$, then one of the following holds:
    \begin{enumerate}
        \item $\cO_E=\cO_{E'}$, 
        \item $[\cO_E:\cO_{E'}]=\ell$,
        \item $[\cO_{E'}:\cO_E]=\ell$.
    \end{enumerate}
\end{proposition}
\begin{proof}
    The result follows from \cite[Proposition~21, p.~44]{kohel}. 
\end{proof}
\noindent An $\ell$-isogeny $\varphi: E\rightarrow E'$ is said to be \emph{horizontal} if $\cO_{E}=\cO_{E'}$. If $[\cO_E:\cO_{E'}]=\ell$ (resp. $[\cO_{E'}:\cO_{E}]=\ell$) then $\varphi$ is said to be \emph{descending} (resp. \emph{ascending}). Given an order $\cO$, let $\op{Ell}_{\cO}(k)$ be the set of isomorphism classes $j$-invariants $j(E)$ for elliptic curves $E_{/k}$ with $\op{End}_{\cO}(k)\simeq \cO$. Given an elliptic curve $E$ with endomorphism ring $\cO$ and a non-zero ideal $\mathfrak{a}$ of $\cO$, we have an isogeny $\varphi_{\mathfrak{a}}:E\rightarrow E'$. Assume that the norm $N(\mathfrak{a})=[\cO:\mathfrak{a}]$ is prime to $p$, then the degree of $\varphi_{\mathfrak{a}}$ equals $N(\mathfrak{a})$. Assume that $\op{Ell}_{\cO}(k)\neq \emptyset$, then the set of elliptic curves $\op{Ell}_{\cO}(k)$ inherits a simply transitive action of the class group $\op{Cl}(\cO)$ of $\cO$. In particular, the cardinality of $\op{Ell}_{\cO}(k)$, when non-zero, is equal to the class number $h(\cO):=\# \op{Cl}(\cO)$. Consider elliptic curves in an ordinary component. An elliptic curve $E$ over $k$ is at depth $d$ if $v_\ell([\cO_K:\cO_E])=d$. If the depth a curve is $0$ then it is at the surface. The height of the volcano is $v_\ell\left([\cO_K: \Z[\pi]]\right)$ where $D_K$ is the discriminant of $K$.
\par Suppose that $\bar{k}=\bar{\F}_q$ be an algebraically closed field and $E_{/\bar{k}}$ be an elliptic curve with
$\operatorname{End}_{\bar{k}}(E)\;\cong\;\mathcal{O}=\mathbb{Z} + f\mathcal{O}_{K}$,
an order of conductor $f$ in the imaginary quadratic field $K$. Let $\ell\neq \operatorname{char}(\bar{k})$ be a prime. In what follows, we refer to $\ell$-isogenies over $\bar{k}$:
\begin{enumerate}
    \item If $\ell\mid f$, then $j(E)$ is not on the crater, and there are no horizontal $\ell$-isogenies. There is a unique ascending $\ell$-isogeny from $j(E)$ and a total of $\ell$ descending isogenies.
    \item If $\ell\nmid f$, then the number of horizontal $\ell$-isogenies equals:
    \[
    \begin{cases}
    0 & \text{if $\ell$ is inert in $K$},\\
    1 & \text{if $\ell$ is ramified in $K$},\\
    2 & \text{if $\ell$ splits in $K$},
    \end{cases}
    \]
    and all remaining $\ell$-isogenies are descending.
\end{enumerate}
Thus, over $\bar{k}$ one has an infinite volcano. For instance, consider the picture depicted below:
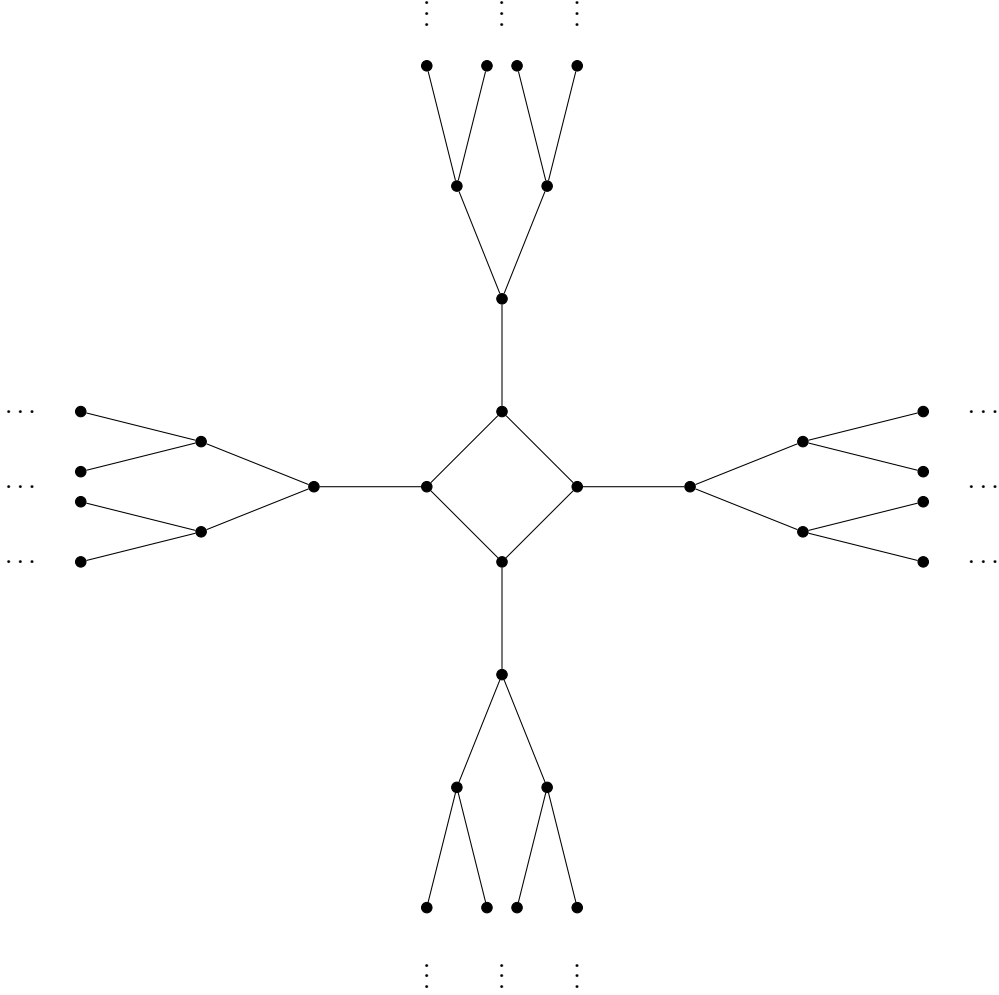
\begin{center}
\begin{tikzpicture}[
  vertex/.style={circle, draw, fill=black, inner sep=1.4pt},
  scale=1
]

\node[vertex] (cN) at (0,1) {};
\node[vertex] (cE) at (1,0) {};
\node[vertex] (cS) at (0,-1) {};
\node[vertex] (cW) at (-1,0) {};

\draw (cN) -- (cE) -- (cS) -- (cW) -- (cN);

\node[vertex] (n1) at (0,2.5) {};
\node[vertex] (e1) at (2.5,0) {};
\node[vertex] (s1) at (0,-2.5) {};
\node[vertex] (w1) at (-2.5,0) {};

\draw (cN) -- (n1);
\draw (cE) -- (e1);
\draw (cS) -- (s1);
\draw (cW) -- (w1);


\node[vertex] (n2l) at (-0.6,4) {};
\node[vertex] (n2r) at (0.6,4) {};
\draw (n1) -- (n2l);
\draw (n1) -- (n2r);

\node[vertex] (e2l) at (4,0.6) {};
\node[vertex] (e2r) at (4,-0.6) {};
\draw (e1) -- (e2l);
\draw (e1) -- (e2r);

\node[vertex] (s2l) at (-0.6,-4) {};
\node[vertex] (s2r) at (0.6,-4) {};
\draw (s1) -- (s2l);
\draw (s1) -- (s2r);

\node[vertex] (w2l) at (-4,0.6) {};
\node[vertex] (w2r) at (-4,-0.6) {};
\draw (w1) -- (w2l);
\draw (w1) -- (w2r);


\node[vertex] (n3ll) at (-1,5.6) {};
\node[vertex] (n3lr) at (-0.2,5.6) {};
\node[vertex] (n3rl) at (0.2,5.6) {};
\node[vertex] (n3rr) at (1,5.6) {};
\draw (n2l) -- (n3ll);
\draw (n2l) -- (n3lr);
\draw (n2r) -- (n3rl);
\draw (n2r) -- (n3rr);

\node[vertex] (e3ll) at (5.6,1) {};
\node[vertex] (e3lr) at (5.6,0.2) {};
\node[vertex] (e3rl) at (5.6,-0.2) {};
\node[vertex] (e3rr) at (5.6,-1) {};
\draw (e2l) -- (e3ll);
\draw (e2l) -- (e3lr);
\draw (e2r) -- (e3rl);
\draw (e2r) -- (e3rr);

\node[vertex] (s3ll) at (-1,-5.6) {};
\node[vertex] (s3lr) at (-0.2,-5.6) {};
\node[vertex] (s3rl) at (0.2,-5.6) {};
\node[vertex] (s3rr) at (1,-5.6) {};
\draw (s2l) -- (s3ll);
\draw (s2l) -- (s3lr);
\draw (s2r) -- (s3rl);
\draw (s2r) -- (s3rr);

\node[vertex] (w3ll) at (-5.6,1) {};
\node[vertex] (w3lr) at (-5.6,0.2) {};
\node[vertex] (w3rl) at (-5.6,-0.2) {};
\node[vertex] (w3rr) at (-5.6,-1) {};
\draw (w2l) -- (w3ll);
\draw (w2l) -- (w3lr);
\draw (w2r) -- (w3rl);
\draw (w2r) -- (w3rr);

\node at (0,6.4) {$\vdots$};
\node at (-1,6.4) {$\vdots$};
\node at (1,6.4) {$\vdots$};
\node at (6.4,0) {$\hdots$};
\node at (6.4,1) {$\hdots$};
\node at (6.4,-1) {$\hdots$};
\node at (0,-6.4) {$\vdots$};
\node at (1,-6.4) {$\vdots$};
\node at (-1,-6.4) {$\vdots$};
\node at (-6.4,0) {$\hdots$};
\node at (-6.4,-1) {$\hdots$};
\node at (-6.4,1) {$\hdots$};
\end{tikzpicture}
\end{center}
\par Now let $k$ be a finite field and let $\mathcal{G}$ be an ordinary connected component of $\mathcal{G}_\ell(k)$, and let $E/k$ be any elliptic curve whose $j$–invariant lies in $\mathcal{G}$. Write $\cO_E=\Z+f_E\cO_K$ for the endomorphism ring of $E$, where $K=\Q(\sqrt{t^2-4q})$. The depth of $E$ in $\mathcal{G}$ is then given by $v_\ell(f_E)$, and this quantity depends only on the connected component containing $E$. Away from the crater, the structure of $\mathcal{G}$ is tree-like. Every vertex not on the crater admits a unique ascending $\ell$–isogeny, corresponding to an inclusion of endomorphism rings of index $\ell$.

\par
From this perspective, ordinary isogeny components may be viewed as finite truncations of an infinite $(\ell+1)$–regular tree, with the truncation occurring at a depth determined by the $\ell$–adic valuation of the discriminant of $\Z[\pi]$. This viewpoint is particularly useful when studying distribution questions in the next section.
\begin{theorem}[Kohel]
    With respect to notation above and assume that $\mathcal{G}=(V,E)$. Then the following assertions hold:
    \begin{enumerate}
        \item The vertices $V_i$ of depth $i$ all have the same endomorphism ring $\cO_i$.
        \item Assume that $V$ does not contain $0$ or $1728$, then $V_0$ has degree $1+\left(\frac{D_0}{\ell}\right)$, where $D_0:=\op{disc}\cO_0$. Further if $\left(\frac{D_0}{\ell}\right)\geq 0$, let $\mathfrak{l}|\ell$ be a prime of $\cO_0$. Then, $|V_0|$ is the order of $[\mathfrak{l}]$ in $\op{Cl}(\cO_0)$. Otherwise, $|V_0|=1$.
        \item The height of $\mathcal{G}$ is \[h(\mathcal{G})=\frac{1}{2}v_\ell\left( \frac{t^2-4p}{D_K} \right)\] where $t$ is the trace of $\pi$ for any elliptic curve in $V$ and $D_K$ is the discriminant of $\cO_K$.
    \end{enumerate}
\end{theorem}
\begin{proof}
    The statement summarizes results proved in \cite{kohel}, see also \cite[Theorem 7]{sutherlandvolcanoes}.
\end{proof}
\noindent In fact, given any volcano $\mathcal{G}$, there does exist prime numbers $\ell$ and $p$ such that $\mathcal{G}$ occurs as a connected component of $\mathcal{G}_\ell(\F_p)$ (see \cite{inverse}).
\textit{Example:} We recall an example here due to Sutherland \cite[Example 9]{sutherlandvolcanoes}. Let $p:=411751$ and $\ell=3$. Set $t:=52$ and consider the isogeny class of elliptic curves $E_{/\F_p}$ with $\#E(\F_p)=t$. There are a total of $1008$ elliptic curves in this isogeny class. Their $j$-invariants make up $10$ $3$-volcanoes in this isogeny class. One has that $4p=t^2-v^2D$ with $v=90=2\times 3^2\times 5$ and $D=-203$. Thus all ten volcanoes have height equal to $2=v_3(90)$. The sizes of the craters consist of $12$ or $4$ vertices.

\section{Variation of volcano height in families of reductions}\label{s3}
In this section we fix an elliptic curve $E/\Q$ and an odd prime $\ell$. Let $N_E$ denote the conductor of $E$. For any prime $p\nmid \ell N_E$, let $E_p$ denote the reduction of $E$ modulo $p$, and define
\[
a_p := p+1-\#E_p(\F_p) \qquad \text{and}\quad \Delta_p := a_p^2-4p.
\]
Let $\pi_p$ denote the Frobenius endomorphism of $E_p$, and write $\Z[\pi_p]$ for the order it generates in $K_p:=\Q(\pi_p)$. Recall that $E_p$ is ordinary if and only if $p\nmid a_p$. In this case, the endomorphism ring $\cO_p := \op{End}(E_p)$
is an order in the imaginary quadratic field $K_p$, and may be written in the form $\cO_p = \Z + f_p \cO_{K_p}$, where $f_p\ge 1$ is the conductor of $\cO_p$.

Let $\mathcal{G}_p$ denote the connected component of the $\ell$-isogeny graph over $\F_p$ containing the vertex $j(E_p)$. Suppose that $p$ is a prime of ordinary reduction. Then the depth of $j(E_p)$ in $\mathcal{G}_p$, denoted $d(p)$, is given by $d(p)=v_\ell(f_p)$, while the height of the corresponding $\ell$-volcano is denoted by $H(p)$. In particular, one has $0 \le d(p) \le H(p)$.

\par By a well known result of Serre, if $E$ is non-CM, then the density of primes with good ordinary reduction is $1$. Our main objective is to determine the natural density (as \(x\to\infty\)) of primes \(p\le x\) for which $p$ is a prime of good ordinary reduction and \(H(p)=r\). When $E_p$ has ordinary reduction, the ring $\Z[\pi_p]$ is an order in the quadratic field $K_p=\Q(\pi_p)$, and we may write \[\Z[\pi_p]=\Z+f_{0,p}\cO_{K_p}\quad \text{and}\quad \cO_p=\Z+f_p\cO_{K_p}.\]Consequently, we find that \[\op{disc}(\Z[\pi_p])=f_{0,p}^{2}\op{disc}(\cO_{K_p})\quad \text{and}\quad \op{disc}(\cO_p)=f_p^{2}\op{disc}(\cO_{K_p}).\]It follows immediately that 
\begin{equation}\label{main disc eqn}\begin{split}
& d(p)=v_\ell(f_{p})=\tfrac12\, v_\ell\!\left(\frac{\op{disc}\cO_p}{\op{disc}\cO_{K_p}}\right),\\
& H(p)=v_\ell(f_{0,p})=\tfrac12\, v_\ell\!\left(\frac{\op{disc}\Z[\pi_p]}{\op{disc}\cO_{K_p}}\right)
   =\tfrac12\, v_\ell\!\left(\frac{a_p^{2}-4p}{\op{disc}\cO_{K_p}}\right).\end{split}
\end{equation}
\noindent Since $v_\ell\left(\op{disc}\cO_{K_p}\right)\in \{0,1\}$, one finds that $H(p)=r$ (resp. $d(p)=r$) if and only if $v_\ell(a_p^2-4p)\in \{2r, 2r+1\}$ (resp. $v_\ell(\op{disc}\cO_p)\in \{2r, 2r+1\}$). We also set $d'(p):=H(p)-d(p)$, which measures the distance of the vertex $j(E_p)$ from the floor. 
\par For an elliptic curve $E_{/\Q}$ and a natural number $n\geq 1$, set $E[n]:=\op{ker}\left\{\times n: E(\bar{\Q})\rightarrow E(\bar{\Q})\right\}$. Set $\op{G}_{\Q}:=\op{Gal}(\bar{\Q}/\Q)$ and note that $E[n]\simeq (\Z/n\Z)^2$ is a module over $\op{G}_{\Q}$. The automorphism group of $E[n]$ is isomorphic to $\op{GL}_2(\Z/n\Z)$. The action of $\op{G}_{\Q}$ on $E[n]$ is encoded by a Galois representation 
\[\rho_{E,n}:\op{G}_{\Q}\rightarrow \op{GL}_2(\Z/n\Z).\] If $m|n$, then multiplication by $(n/m)$ gives a surjective $\op{G}_{\Q}$-equivariant map $\pi_{n,m}:E[n]\rightarrow E[m]$. The mod-$m$ reduction of $\rho_{E,n}$ is then identified with $\rho_{E,m}$. The \emph{big Tate module} $\mathbb{T}_E$ is the Galois module $\varprojlim_n E[n]$, where the inverse limit is taken with respect to the maps $\pi_{n, m}$ defined above. We choose compatible bases $E[n]\simeq (\Z/n\Z)\cdot P_{1}^n\oplus (\Z/n\Z)\cdot P_{2}^n$, such that $\pi_{n,m}$ maps $P_i^n$ to $P_i^m$ for $i=1,2$. Let $P_i\in \mathbb{T}_E$ be the inverse limit $(P_i^n)_n$. Then, $\mathbb{T}_E=\widehat{\Z}P_1\oplus \widehat{\Z}P_2$ and its automorphism group is $\op{GL}_2(\widehat{\Z})$. The Galois representation on $\mathbb{T}$ is denoted 
\[\widehat{\rho}_E:\op{G}_{\Q}\rightarrow \op{GL}_2(\widehat{\Z}),\]
and in the literature is often referred to as the \emph{adelic Galois representation}. Given a prime number $\ell$, let $T_\ell(E):=\varprojlim_n E[p^n]$ be the \emph{$\ell$-adic Tate module} and let 
\[\widehat{\rho}_{E, \ell}: \op{G}_{\Q}\rightarrow \op{GL}_2(\Z_\ell)\] be the associated Galois representation. There is a natural decomposition $\mathbb{T}\simeq \prod_\ell T_\ell(E)$ and identify $\widehat{\rho}_E$ with the product $\prod_\ell \widehat{\rho}_{E, \ell}$. Given a natural number $n>0$, let $\Q(E[n])$ be the field \emph{cut out} by $\rho_{E, n}$. In other words, $\Q(E[n])$ is the field fixed by the kernel of $\rho_{E,n}$. By the Galois correspondence, $\op{Gal}(\Q(E[n])/\Q)$ can be identified with the image of $\rho_{E,n}$ in $\op{GL}_2(\Z/n\Z)$. In particular, if $\rho_{E,n}$ is surjective then the Galois group $\op{Gal}(\Q(E[n])/\Q)$ is isomorphic to $\op{GL}_2(\Z/n\Z)$. By the Neron--Ogg--Shafarevich criterion, the primes $p\nmid nN_E$ are unramified in $\Q(E[n])$. For each such prime $p$, choose a prime $\mathfrak{p}|p$ of $\Q(E[n])$ that lies above $p$. Let \[\op{Frob}_p:=\op{Frob}_{\mathfrak{p}}\in \op{Gal}(\Q(E[n])/\Q)\] be the associated Frobenius element. The conjugacy class generated by $\op{Frob}_p$ is independent of the choice of $\mathfrak{p}$.

\begin{theorem}[Serre's open image theorem \cite{serre1, serre2}]
    Suppose that $E$ does not have complex multiplication. Then, the image of $\widehat{\rho}_E$ is of finite index in $\op{GL}_2(\widehat{\Z})$. 
\end{theorem}
\noindent It follows from the above that for a non-CM elliptic curve $E$, the $\ell$-adic representation $\widehat{\rho}_{E, \ell}$ is surjective for all but finitely many primes $\ell$. A prime number $\ell$ is said to be \emph{exceptional} if $\widehat{\rho}_{E, \ell}$ is not surjective. The index $\delta_E:=[\op{GL}_2(\widehat{\Z}):\op{image} \widehat{\rho}_E]$ is even, and $E$ is said to be a \emph{Serre curve} if $\delta_E=2$. If $E$ is a Serre curve, there are no exceptional primes. 

\par Write $E=E_{A,B}:y^2=x^3+Ax+B$ where $(A,B)\in \Z^2$. Then $E$ is said to be minimal if there is no prime $p$ such that $p^4|A$ and $p^6|B$. The \emph{naive height} of $E$ is then $h(E):=\op{max} \{|A|^3, |B|^2\}$. Let $X>0$ be a real number and let $\cC(X)$ denote the family of elliptic curves $E_{A,B}$ of height at most $X$, namely
\[
\cC(X):=\{(A,B)\in \cC : h(E_{A,B}) \le X\}.
\]
\begin{definition}
Any set of isomorphism classes of elliptic curves over $\Q$ may be identified with a subset $S\subseteq \cC$. For $X>0$, define $S(X):=S\cap \cC(X)$. The density of $S$ (if it exists) is defined to be
\[
\lim_{X\to\infty} \frac{\# S(X)}{\# \cC(X)}.
\]
\end{definition}
\par Duke \cite{dukeexceptional} proved that the set of elliptic curves $E_{/\Q}$ with no exceptional primes has density $1$. In other words, almost all elliptic curves have no exceptional primes. Jones \cite{jonesalmostall} then refined this result to show that almost all elliptic curves are Serre curves. First we consider elliptic curves for which the $\ell$-adic Galois representation is surjective. By the result of Duke above, this assumption is satisfied for almost all elliptic curves.

\par Let $E_{/\Q}$ be an elliptic curve without complex multiplication and let $\ell$ be an odd prime. Let $r\geq 0$ be an integer. We prove three results in this section.
\begin{enumerate}
    \item Assume that $\widehat{\rho}_{E, \ell}$ is surjective. Theorem \ref{section 3 thm 1} gives the density of primes $p\neq \ell$ at which $E$ has good ordinary reduction and the volcano height $H(p)=r$. 
    \item For an elliptic curve as in part (1), we compute the density of primes $p\neq \ell$ at which $E$ has good ordinary reduction and distance from the floor $d'(p)=r$ (see Theorem \ref{section 3 thm 2}).
    \item Now assume that $E$ is any non-CM elliptic curve. Theorem \ref{section 3 thm 3} gives the desnity of primes $p\neq \ell$ for which $H(p)=d'(p)=r$ is positive.
\end{enumerate}
\noindent We interpret the conditions $H(p)=r$ and $d'(p)=r$ in terms of the image of Frobenius $\op{Frob}_p$ with respect to the $\ell$-adic Galois representation. After careful analysis of these conditions, the results then follow from an application of the Chebotarev density theorem.
\par We prove a counting result for matrices in $\op{GL}_2(\Z/\ell^n\Z)$ whose discriminant has a prescribed $\ell$-adic valuation. For $x\in \Z/\ell^n \Z$, write $x=\ell^t x'$ where $\ell\nmid x'$, where $t\in [0, n]$, with the convention that $t=n$ if $x=0$. We then define $v_\ell(x):=t$.
\begin{proposition}\label{section 3 main prop}
Let $\ell$ be an odd prime and let $n,r\geq 1$ be integers with $n\ge 2r$. Setting
$R=\Z/\ell^n\Z$, for $M=\begin{pmatrix}a&b\\ c&d\end{pmatrix}\in M_2(R)$ define
\[F(a,b,c,d):=(\tr M)^2-4\det M=(a-d)^2+4bc.
\]
Letting
\[
S_{n,r}:=\{M\in\op{GL}_2(R):\,v_\ell(F(a,b,c,d))\ge 2r\},
\]
we have that
\begin{equation}\label{S n r density}
\frac{|S_{n,r}|}{|\mathrm{GL}_2(\mathbb{Z}/\ell^n)|} 
= 2\ell^{-2r}
+
\frac{\ell^{-2r}\big(2+2\ell-(2r+2)\ell^{1-r}\big)}{\ell^{2}-1}.
\end{equation}
where the implied constant in $O_r$ depends on $r$ but not on $\ell$.
\end{proposition}

\begin{proof}
Make the invertible $R$–linear change of variables
\[
x=a-d,\qquad y=b,\qquad z=c,\qquad w=a+d,
\]
whose Jacobian is a unit in $R$ since $\ell$ is odd. Then with respect to the variables $(x,y,z,w)$, we find that
\[
F(a,b,c,d)=x^2+4yz,
\qquad\text{and}\quad
\det M=\frac{w^2-x^2}{4}-yz.
\]
We write $F(x,y,z):=F(a,b,c,d)=x^2+4yz$. The condition $v_\ell\left(F(a,b,c,d)\right)\ge2r$ defining the set $S_{n,r}$ depends only on $(x,y,z)$ and not on $w$. The congruence $x^2+4yz\equiv0\pmod{\ell^{2r}}$ implies that $\det M\equiv w^2/4\pmod{\ell^{2r}}$ and thus $\det M\in R^\times$ if and only if $w\not\equiv0\pmod\ell$.  Hence each
$(x,y,z)\in R^3$ satisfying $F(x,y,z)\equiv 0\pmod{\ell^{2r}}$ contributes to exactly $\ell^{\,n-1}(\ell-1)$ choices of $w$. Therefore, we find that
\[|S_{n,r}|=\ell^{\,n-1}(\ell-1)\#\{(x,y,z)\in R^3\mid F(x,y,z)\equiv 0\pmod{\ell^{2r}}\}.\]
\noindent The congruence $F(x,y,z)\equiv0\pmod{\ell^{2r}}$ depends only on residues
modulo $\ell^{2r}$, so the number of solutions modulo $\ell^n$ is
$\ell^{3(n-2r)}$ times the number
\[
N_0:=\#\{(x,y,z)\in(\Z/\ell^{2r}\Z)^3|\ F(x,y,z)\equiv0\pmod{\ell^{2r}}\}.
\] 
\noindent In other words,
\[|S_{n,r}|=\ell^{\,n-1}(\ell-1)\ell^{3(n-2r)} N_0=\ell^{4n-6r-1}(\ell-1)N_0.\]
\noindent Write $t=-x^2/4$ and $u=v_\ell(x)\in [0,2r]$. Thus
\begin{equation}\label{N_0 eqn}N_0=\sum_{u=0}^{2r} \#\left\{(x,y,z)\in(\Z/\ell^{2r}\Z)^3|\, v_\ell(x)=u\,\text{and}\, yz\equiv -\frac{x^2}{4}\pmod{\ell^{2r}}\right\}.\end{equation}
Suppose that $u<r$, the congruence 
\[yz\equiv -\frac{x^2}{4}\pmod{\ell^{2r}}\]determines $x^2\in \Z/\ell^{2r}\Z$. Writing $x=\ell^ux'$ where $x'$ is an element in $\Z/\ell^{2r-u}$ which is not divisible by $\ell$, we find that $(x')^2\pmod{\ell^{2r-2u}}$ is determined. Thus for any pair $(y,z)$ with $v_\ell(yz)=u$, the number of choices of $x$ is $2\ell^u$. Thus, we find that for $u<r$,
\begin{equation}\label{u<r eqn}\begin{split}&\#\left\{(x,y,z)\in(\Z/\ell^{2r}\Z)^3|\, v_\ell(x)=u\,\text{and}\, yz\equiv -\frac{x^2}{4}\pmod{\ell^{2r}}\right\}\\
=&2\ell^u\#\left\{(y,z)\in(\Z/\ell^{2r}\Z)^2|\, v_\ell(y)+v_\ell(z)=2u\right\}\\
=& 2\ell^u\sum_{i+j=2u} \#\left\{(y,z)\in(\Z/\ell^{2r}\Z)^2|\, v_\ell(y)=i\quad\text{and}\quad v_\ell(z)=j\right\}\\
=& 2\ell^u\sum_{i+j=2u} \varphi(\ell^{2r-i})\varphi(\ell^{2r-j})\\
=& 2\ell^u\sum_{i+j=2u} \ell^{4r-2u}\left(1-\ell^{-1}\right)^2\\
=& 2(2u+1) \ell^{4r-u}\left(1-\ell^{-1}\right)^2.
\end{split}\end{equation}
Next, consider the case when $u\geq r$ and thus $x^2\equiv 0\pmod{\ell^{2r}}$. In this case, we find that 
\[\begin{split}&\#\left\{(x,y,z)\in(\Z/\ell^{2r}\Z)^3|\, v_\ell(x)=u\,\text{and}\, yz\equiv 0\pmod{\ell^{2r}}\right\}\\
=&\sum_{\substack{0\leq i,j\leq 2r\\
i+j\geq 2r}} \#\left\{(x,y,z)\in(\Z/\ell^{2r}\Z)^2|v_\ell(x)=u,\quad \, v_\ell(y)=i\quad\text{and}\quad v_\ell(z)=j\right\}\\
=& \sum_{\substack{0\leq i,j\leq 2r\\
i+j\geq 2r}}\varphi(\ell^{2r-i})\varphi(\ell^{2r-j})\varphi(\ell^{2r-u}).\\
\end{split}\]
\noindent Changing variables \(p:=2r-i,\; q:=2r-j\), we find that \(p,q\in\{0,\dots,2r\}\) and the condition \(i+j\ge2r\) is equivalent to \(p+q\le2r\). Hence
\[
\sum_{\substack{0\le i,j\le 2r\\ i+j\ge 2r}}
\varphi(\ell^{2r-i})\varphi(\ell^{2r-j})
=
\sum_{p=0}^{2r}\sum_{q=0}^{2r-p}\varphi(\ell^{p})\varphi(\ell^{q}).
\]
\noindent For \(m\ge1\) we have \(\varphi(\ell^{m})=\ell^{m}-\ell^{m-1}=\ell^{m-1}(\ell-1)\), and \(\varphi(1)=1\). Set
\(
a_m:=\varphi(\ell^m)
\)
and denote the partial sums \(S_m:=\sum_{q=0}^{m}a_q\). Clearly, $S_0=1=\ell^0$ and for $m\geq 1$, one has that
\[
S_m=1+\sum_{q=1}^m \ell^{q-1}(\ell-1)=1+(\ell-1)\frac{\ell^{m}-1}{\ell-1}=\ell^{m}.
\]
Consequently,
\[
\begin{split} \sum_{p=0}^{2r}\sum_{q=0}^{2r-p}a_p a_q
=&\sum_{p=0}^{2r} a_p\Big(\sum_{q=0}^{2r-p}a_q\Big)
=\sum_{p=0}^{2r} a_p\,\ell^{2r-p}\\
=&\ell^{2r} + \sum_{p=1}^{2r} \ell^{p-1}(\ell-1)\,\ell^{2r-p}\\
= & \ell^{2r} + \sum_{p=1}^{2r} \ell^{2r-1}(\ell-1)\\
=&\ell^{2r-1}\big((2r+1)\ell - 2r\big).
\end{split}
\]
We have shown that for $u\geq r$,
\begin{equation}\label{u>=r eqn}\begin{split}
&\#\left\{(x,y,z)\in(\Z/\ell^{2r}\Z)^3|\, v_\ell(x)
=u\,\text{and}\, yz\equiv 0\pmod{\ell^{2r}}\right\}\\
=&\begin{cases}
\ell^{2r-1}\big((2r+1)\ell-2r\big), & \text{if }u=2r,\\
\ell^{4r-u-2}(\ell-1)\big((2r+1)\ell-2r\big), & \text{if }r\le u\le 2r-1.
\end{cases}
\end{split}
\end{equation}
Substituting \eqref{u<r eqn} and \eqref{u>=r eqn} into \eqref{N_0 eqn}, we find that
\[N_0=\sum_{u=0}^{r-1}2(2u+1) \ell^{4r-u}\left(1-\ell^{-1}\right)^2+\sum_{u=r}^{2r-1}\ell^{4r-u-2}(\ell-1)\big((2r+1)\ell-2r\big)+\ell^{2r-1}\big((2r+1)\ell-2r\big).\]
\noindent Set 
\[\begin{split}
&S_1:=\sum_{u=0}^{r-1}2(2u+1) \ell^{4r-u}\left(1-\ell^{-1}\right)^2,\\
&S_2:=\sum_{u=r}^{2r-1}\ell^{4r-u-2}(\ell-1)\big((2r+1)\ell-2r\big),\\
&S_3:=\ell^{2r-1}\big((2r+1)\ell-2r\big).
\end{split}\]
For the first sum above, consider the following reductions:
\[
\begin{split}\sum_{u=0}^{r-1} (2u+1)\ell^{-u}(1-\ell^{-1})^{2}
= &(1-\ell^{-1})^{2}\Bigl(2\sum_{u=0}^{r-1}u\ell^{-u}+\sum_{u=0}^{r-1}\ell^{-u}\Bigr)\\
= &(1-\ell^{-1})^{2}\left( 
2\cdot\frac{\ell^{-1}(1-r\ell^{-r-1}+(r-1)\ell^{-r})}{(1-\ell^{-1})^{2}}
+ \frac{1-\ell^{-r}}{1-\ell^{-1}}
\right)\\
= &2\ell^{-1}(1-r\ell^{-r-1}+(r-1)\ell^{-r}) + (1-\ell^{-r})(1-\ell^{-1})\\
= &1+\ell^{-1}-(2r+1)\ell^{-r}+(2r-1)\ell^{-r-1}.
\end{split}
\]
\noindent Multiply by $2 \ell^{4r}$ to deduce that
\[
S_1=\sum_{u=0}^{r-1}2(2u+1) \ell^{4r-u}\left(1-\ell^{-1}\right)^2= 2\ell^{4r} + 2\ell^{4r-1} - 2(2r+1)\ell^{3r} + 2(2r-1)\ell^{3r-1}.
\]
\noindent For the second sum, set $A = (2r+1)\ell - 2r$. We find that
\[
S_2=\sum_{u=r}^{2r-1} \ell^{4r-u-2} (\ell-1) A 
= A (\ell-1) \sum_{u=r}^{2r-1} \ell^{4r-u-2}=A \ell^{3r-1} - A \ell^{2r-1}.
\]
Thus one finds that
\[
S_2 + S_3 = A \ell^{3r-1} - A \ell^{2r-1} + A \ell^{2r-1} = A \ell^{3r-1}.
\]
Combining the above, we deduce that
\[
N_0= S_1 + S_2 + S_3 =  2\ell^{4r} + 2\ell^{4r-1} - (2r+1)\ell^{3r} - \ell^{3r-1}.
\]
and 
\[
|S_{n,r}| = \ell^{4n-6r-1}(\ell-1)N_0=(\ell-1)\, \ell^{4n-3r-2} \Big( 2 \ell^{r+1} + 2 \ell^r - (2r+1)\ell -1 \Big)
\]
Thus, we have shown that
\[\begin{split}
\frac{|S_{n,r}|}{|\op{GL}_2(\mathbb{Z}/\ell^n)|} 
=&\frac{(\ell-1)\,\ell^{4n-3r-2}\big(2\ell^{r+1}+2\ell^{r}-(2r+1)\ell-1\big)}
       {\ell^{4n-3}(\ell-1)^2(\ell+1)}\\
=&\frac{\ell^{\,1-3r}\big(2\ell^{r+1}+2\ell^{r}-(2r+1)\ell-1\big)}
       {(\ell^2-1)}\\
=&
2\ell^{-2r}
+
\frac{\ell^{-2r}\big(2+2\ell-(2r+2)\ell^{1-r}\big)}{\ell^{2}-1}.
\end{split}
\]
\end{proof}

\begin{theorem}\label{section 3 thm 1}
    Let $\ell$ be an odd prime such that $\ell\neq \op{char}k$ and assume that $E$ is a non-CM elliptic curve for which $\widehat{\rho}_{E, \ell}:\op{G}_{\Q}\rightarrow \op{GL}_2(\Z_\ell)$ is surjective. Then for $r>0$, the density of primes $p$ at which $E$ has good ordinary reduction and $H(p)=r$ is precisely 
    \[\begin{split}&2\ell^{-2r}(1-\ell^{-2})
+
\frac{\ell^{-2r}}{\ell^{2}-1}
\Big(
2+2\ell-2\ell^{-2}-2\ell^{-1}
-(2r+2)\ell^{1-r}
+(2r+4)\ell^{-r-2}
\Big)\\
=& 2\ell^{-2r}+O_r(\ell^{-2r-1}),\end{split}\]where the implied constant in $O_r$ depends on $r$ but not on $\ell$. On the other hand, if $r=0$, the density equals 
\[
1-2\ell^{-2}-\frac{2\ell^{-2}}{\ell+1}.
\]

\end{theorem}

\begin{proof}
Let $p$ be a prime at which $E$ has good ordinary reduction and let $n\geq 2r$. By the Neron--Ogg-Shafarevich criterion, the representation $\rho_{E, \ell^n}:\op{G}_{\Q}\rightarrow \op{GL}_2(\Z/\ell^n\Z)$ is unramified at $p$. Let $\sigma_p:=\rho_{E, \ell^n}(\op{Frob}_p)$ be the image of the Frobenius element. The characteristic polynomial of $\sigma_p$ is 
    \[\op{ch}(\sigma_p, T)=T^2-a_p(E)T+p\]\noindent and its discriminant is $a_p^2-4p$. We note that this characteristic polynomial is independent of the choice of Frobenius element. Recall from \eqref{main disc eqn} that $H(p)=r$ if and only if $v_\ell(a_p^2-4p)\in \{2r, 2r+1\}$. Consequently, $H(p)=r$ if and only if $p\in S_{n, r}\setminus S_{n, r+1}$. Note that the set $S_{n, r}\setminus S_{n, r+1}$ is stable under conjugation. The homomorphism $\rho_{E,n}$ gives rise to an isomorphism 
    \[\rho:\op{Gal}(\Q(E[\ell^n]/\Q)\xrightarrow{\sim} \op{GL}_2(\Z/\ell^n\Z),\] and let $T_{n,r}:=\rho^{-1}\left(S_{n, r}\setminus S_{n, r+1}\right)$. 
    
    \par Let $p\neq \ell$ be a prime of good ordinary reduction. As a consequence of the above discussion, $H(p)=r$ if and only if $\op{Frob}_p$ is contained in the conjugation stable set $T_{n,r}$.  A result of Serre shows that the set of primes $p$ at which $E$ has good ordinary reduction is $1$. By the Chebotarev density theorem, the density of such primes $p\neq \ell$ for which $H(p)=r$ is equal to \[\frac{|T_{n,r}|}{|\op{GL}_2(\Z/\ell^n\Z)|}=\frac{|S_{n,r}|}{|\op{GL}_2(\Z/\ell^n\Z)|}-\frac{|S_{n,r+1}|}{|\op{GL}_2(\Z/\ell^n\Z)|}.\]
    \noindent From \eqref{S n r density}, we find that for $r>0$,
    \[\frac{|T_{n,r}|}{|\op{GL}_2(\Z/\ell^n\Z)|}=2\ell^{-2r}(1-\ell^{-2})
+
\frac{\ell^{-2r}}{\ell^{2}-1}
\Big(
2+2\ell-2\ell^{-2}-2\ell^{-1}
-(2r+2)\ell^{1-r}
+(2r+4)\ell^{-r-2}
\Big).
\]
On the other hand, if $r=0$ we have that 
\[\frac{|T_{n,0}|}{|\op{GL}_2(\Z/\ell^n\Z)|}=1-\frac{|S_{n,1}|}{|\op{GL}_2(\Z/\ell^n\Z)|}=1-2\ell^{-2}-\frac{2\ell^{-2}}{\ell+1}.
\]
\noindent This completes the proof.
\end{proof}
\par Setting $b_p:=[\cO_p:\Z[\pi_p]]=f_p/f_{0,p}$, a result of Duke and Toth \cite{duketoth} states that for any integer $n$ which is coprime to $p$,
$\rho_{E,n}(\op{Frob}_p)$ is conjugate to the mod-$n$ reduction of the matrix 
\begin{equation}\label{duketoth matrix}\mtx{\frac{a_p+b_p\delta_p}{2}}{b_p}{b_p\frac{(\Delta_p-\delta_p)}{2}}{\frac{a_p-b_p\delta_p}{2}},\end{equation} where $\Delta_p$ is the discriminant of $\cO_p$ and $\delta_p=0,1$ according to as to $a_p\equiv 0,1\pmod{2}$. 

\begin{theorem}\label{section 3 thm 2}
    Let $\ell$ be an odd prime such that $\ell\neq \op{char}k$ and assume that $E$ is a non-CM elliptic curve for which $\widehat{\rho}_{E, \ell}:\op{G}_{\Q}\rightarrow \op{GL}_2(\Z_\ell)$ is surjective. Then for $r>0$, the density of primes $p$ at which $E$ has good ordinary reduction and $d'(p)=r$ is precisely $\frac{\ell^2+\ell+1}{(\ell+1)\ell^{3r+1}}$. The density of primes $p$ for which $d'(p)=0$ equals $\frac{\ell^3-\ell-1}{\ell(\ell-1)(\ell+1)}$.
\end{theorem}
\begin{proof}
    By a well known result of Serre, the set of primes $p$ at which $E$ has good ordinary reduction has density $1$. Assume that $r>0$ and let $D_{r}$ be the set of matrices in $\op{GL}_2(\Z/\ell^{r+1})$ which are not scalar but are scalar modulo $\ell^r$. It is easy to see that 
    \[|D_r|=(\ell^{r-1}(\ell-1))(\ell^4-\ell).\]
    Let $p\neq \ell$ be a prime of good ordinary reduction and write $b_p=\ell^v b'$ where $v:=v_\ell(b_p)$. Then since $\rho_{E,\ell^n}(\op{Frob}_p)$ is conjugate to the mod-$\ell^n$ reduction of the matrix \eqref{duketoth matrix}, it follows that $v_\ell(b_p)$ equals the smallest value of $k$ such that $\rho_{E, \ell^k}(\op{Frob}_p)$ is scalar. In other words, $d'(p)=r$ if and only if $\rho_{E,\ell^{r+1}}(\op{Frob}_p)\in D_{r}$. It is assumed that $\widehat{\rho}_{E, \ell}$ is surjective, and consequently, 
    \[\op{Gal}(\Q(E[\ell^{r+1}])/\Q)\xrightarrow{\sim} \op{GL}_2(\Z/\ell^{r+1})\] via the isomorphism induced by $\rho_{E, \ell^{r+1}}$. Let $\Sigma_r$ be the set of elements $\sigma\in \op{Gal}(\Q(E[\ell^{r+1}])/\Q)$ which map to $D_r$. Note that $D_r$ is stable under conjugation, and hence so is $\Sigma_r$. In other words, $\Sigma_r$ is a union of conjugacy classes. By the Chebotarev density theorem applied to the Galois extension $\Q(E[\ell^{r+1}])/\Q$, the density of primes for which $d'(p)=r$ equals
    \[\frac{|\Sigma_{r}|}{[\Q(E[\ell^{r+1}]):\Q]}=\frac{|D_{r}|}{|\op{GL}_2(\Z/\ell^{r+1})|}=\frac{(\ell^{r-1}(\ell-1))(\ell^4-\ell)}{(\ell^2-1)(\ell^2-\ell)\ell^{4r}}=
\frac{\ell^2+\ell+1}{(\ell+1)\ell^{3r+1}}.
\]
Let $\Sigma_0$ be the set of elements $\sigma\in \op{Gal}(\Q(E[\ell])/\Q)$ which map to a nonscalar matrix in $\op{GL}_2(\Z/\ell)$. By the Chebotarev density theorem, the density of primes for which $d'(p)=0$ equals
 \[\frac{|\Sigma_{0}|}{[\Q(E[\ell^{r+1}]):\Q]}=1-\frac{(\ell-1)}{(\ell^2-1)(\ell^2-\ell)}=\frac{\ell^3-\ell-1}{\ell(\ell-1)(\ell+1)}.\]
\end{proof}

\par Next consider a non-CM elliptic curve $E_{/\Q}$ and and suppose that $\ell$ is an exceptional prime number. In other words, $\widehat{\rho}_{E, \ell}$ is not surjective. In this case by Serre's open image theorem, there is a minimal integer $k\geq 1$ such that the image of $\widehat{\rho}_{E,\ell}$ contains $\mathcal{G}^k:=\op{ker}\left(\op{GL}_2(\Z_\ell)\longrightarrow \op{GL}_2(\Z/\ell^k)\right)$. 
\begin{theorem}\label{section 3 thm 3}
    Let $\ell\neq p$ be an odd prime number and assume that $E$ is a non-CM elliptic curve over $\Q$. Let $k\geq 1$ be as above and $r\geq k$. Then, the density of primes $p\neq \ell$ for which $E$ has good ordinary reduction and $H(p)=d'(p)=r$ is positive. 
\end{theorem}
\begin{proof}
    Consider the matrix $M_r:=\mtx{1}{\ell^r}{\ell^r}{1}\in \op{GL}_2(\Z/\ell^{2r+1})$. By assumption, $M_r$ is in the image of $\rho_{E, \ell^{2r+1}}$ and let $C$ be the conjugacy class of $\op{Gal}(\Q(E[\ell^{2r+1}])/\Q)$ consisting of $\sigma$ such that $\rho_{E, \ell^{2r+1}}(\sigma)$ is conjugate to $M_r$. If $p$ is a prime of good ordinary reduction such that $\op{Frob}_p\in C$ then the discriminant of $\rho_{E, \ell^{2r+1}}(\op{Frob}_p)$ equals $4\ell^{2r}\pmod{\ell^{2r+1}}$. As a result, $H(p)=r$. On the other hand, since $\rho_{E, \ell^{r}}(\op{Frob}_p)$ is a scalar matrix and $\rho_{E, \ell^{r+1}}(\op{Frob}_p)$ is not a scalar matrix, it follows that $d'(p)=r$ and the result follows from the Chebotarev density theorem applied to the Galois extension $\Q(E[\ell^{2r+1}])/\Q$.
\end{proof}

Let now $E_{/\Q}$ be an elliptic curve with complex multiplication, and write $\cO := \op{End}_{\bar{\Q}}(E)$
for its geometric endomorphism ring. Then $\cO$ is an order of conductor $f$ in an imaginary quadratic field $K$. Let $N_E$ denote the conductor of $E$, and fix a prime $\ell$. For any prime $p\nmid N_E$, the curve $E$ has good reduction at $p$, and the reduction $E_p$ is either ordinary or supersingular. By the theory of complex multiplication, the reduction type is governed by the splitting behavior of $p$ in $K$: the prime $p$ is split in $K$ if and only if $E_p$ is ordinary, while $p$ is inert in $K$ if and only if $E_p$ is supersingular. In particular, $E$ has good ordinary reduction at precisely those primes $p$ which split in $K$.

\par Assume now that $p$ is a prime of good ordinary reduction which is coprime to $f$. The natural reduction homomorphism
\[
\op{End}_{\overline{\Q}}(E)\;\rightarrow\;\cO_p := \op{End}(E_p)
\]
is always injective. Moreover, since $p$ is unramified in $K$ and does not divide the conductor of $\cO$, the theory of complex multiplication implies that no new endomorphisms arise upon reduction. Consequently, this map is an isomorphism, and we obtain an identification $\cO \;\cong\; \cO_p$. It follows that for all but finitely many primes $p$ of good ordinary reduction, the endomorphism ring of $E_p$ coincides with $\cO$. Since the values $H(p)$ and $d'(p)$ depend only on $\cO_p$, it follows that these functions are eventually constant (for ordinary primes $p$). This is in stark contrast to the non-CM case.

\section{Hurwitz class numbers and average volcano depth}\label{s4}

\par
Fix now an odd prime $\ell\neq p$ and an integer $r\ge 0$. Let $\mathcal{E}(\F_q)$ denote the set of $\F_q$-isomorphism classes of elliptic curves over $\F_q$, and let $\mathcal{E}(r;\F_q)\subseteq \mathcal{E}(\F_q)$ be the subset consisting of ordinary elliptic curves whose associated $\ell$–isogeny graph has height exactly $r$. Writing $q=p^k$, we extend the Legendre symbol by setting $\left(\frac{x}{q}\right):=\left(\frac{x}{p}\right)^k$.

\par
The total number of elliptic curves over $\F_q$ up to isomorphism is given by
\[
\#\mathcal{E}(\F_q)=2q+3+2\left(\frac{-3}{q}\right)+\left(\frac{-4}{q}\right)
=2q+O(1),
\]
see \cite[Proposition~5.7]{Schoof}. Our goal in this section is to study the asymptotic behavior of $\#\mathcal{E}(r;\F_q)$ as $q\to\infty$, and in particular the limiting density
\[
\mathfrak{d}_r
:=\lim_{q\to\infty}\frac{\#\mathcal{E}(r;\F_q)}{\#\mathcal{E}(\F_q)}
=\lim_{q\to\infty}\frac{\#\mathcal{E}(r;\F_q)}{2q}.
\]

\par
In order to analyze $\#\mathcal{E}(r;\F_q)$, it is convenient to organize elliptic curves over $\F_q$ according to their Frobenius traces. Recall that for an elliptic curve $E/\F_q$ with Frobenius endomorphism $\pi_E$, the number of $\F_q$–rational points is given by
\[
\#E(\F_q)=q+1-t,
\qquad t:=\operatorname{tr}(\pi_E),
\]
and that two elliptic curves over $\F_q$ are isogenous if and only if they have the same trace $t$. We write $N(t)$ for the number of $\F_q$–isomorphism classes of elliptic curves with trace of Frobenius equal to $t$.

\par
The discriminant $\Delta:=t^2-4q$
governs the arithmetic of the isogeny class corresponding to $t$. When $t^2<4q$ and $p\nmid t$, the isogeny class is ordinary, and the endomorphism ring of any curve in the class is an order in the imaginary quadratic field $\Q(\sqrt{\Delta})$. A fundamental theorem of Deuring identifies the quantity $N(t)$ with the Hurwitz class number $H(\Delta)$ associated to this discriminant. This correspondence allows one to translate questions about the distribution of elliptic curves over finite fields into problems about averages of Hurwitz class numbers over congruence conditions on $t$. In particular, since an ordinary elliptic curve $E/\F_q$ lies in $\mathcal{E}(r;\F_q)$ if and only if
\begin{equation}\label{v_ell criterion}
v_\ell(t^2-4q)\in\{2r,2r+1\},
\end{equation}
the quantity $\#\mathcal{E}(r;\F_q)$ may be expressed as a difference of sums of Hurwitz class numbers. This observation forms the starting point for the asymptotic analysis that follows.

\par We briefly discuss Kronecker (or Hurwitz) class number and their relationship to the number of elliptic curves over a given finite field in an isogeny class. Standard references include Waterhouse's thesis \cite{waterhouse} and work of Deuring \cite{deuring1, deuring2}. Let $\Delta$ be a negative integer with $\Delta\equiv 0,1\pmod{4}$ and let $B(\Delta)$ be the set of integral quadratic forms $aX^2+bXY+cY^2\in \Z[X,Y]$ such that $a>0$ and discriminant $b^2-4ac=\Delta$. A quadratic form is \emph{primitive} if $\op{gcd}(a,b,c)=1$. Denote by $b(\Delta)$ the set of primitive quadratic forms in $B(\Delta)$. The group $\op{SL}_2(\Z)$ acts on both $B(\Delta)$ and $b(\Delta)$. The \emph{class numbers}
\[H(\Delta):=\# \left(B(\Delta)/\op{SL}_2(\Z)\right)\quad \text{and}\quad h(\Delta):=\# \left(b(\Delta)/\op{SL}_2(\Z)\right)\]are related as follows:
\[\sum_d h\left(\frac{\Delta}{d^2}\right)=H(\Delta)\]
\noindent where $d$ runs over positive integers such that $d^2|\Delta$ and $\Delta/d^2\equiv 0,1\pmod{4}$.
\begin{theorem}\label{N(t)= thm}
    If $t^2<4q$ and $p\nmid t$, we have that 
    \[N(t)=H(t^2-4q).\]
\end{theorem}
\begin{proof}
    For a proof of this result, see \cite[Theorem 4.6]{Schoof}.
\end{proof}
As a consequence of Theorem \ref{N(t)= thm} and \eqref{v_ell criterion}, one has that 
\[\begin{split}\# \mathcal{E}(r;\F_q)=&\sum_{\substack{|t|\leq 2\sqrt{q}\\
t^2\equiv 4q\pmod{\ell^{2r}}}} N(t)-\sum_{\substack{|t|\leq 2\sqrt{q}\\
t^2\equiv 4q\pmod{\ell^{2r+2}}}} N(t)\\
=& \sum_{\substack{|t|\leq 2\sqrt{q}\\
t^2\equiv 4q\pmod{\ell^{2r}}}} H(t^2-4q)-\sum_{\substack{|t|\leq 2\sqrt{q}\\
t^2\equiv 4q\pmod{\ell^{2r+2}}}} H(t^2-4q).
\end{split}\]

\begin{theorem}[Hurwitz]\label{thm:hurwitz}
Let $N\ge 1$ be an integer and let $p\nmid N$ be a prime. Fix an integer $a$, and define
\[
\chi(N):=\left(\frac{a^2-4p}{N}\right),
\qquad
\delta:=\frac{N+\chi(N)}{N^2-1}.
\]
Then one has the asymptotic formula
\[
\sum_{t\equiv a \!\!\!\pmod{N}} H(t^2-4p)
=2\delta\, p+O\!\left(N\sqrt{p}\right),
\]
where the implied constant is absolute.
\end{theorem}

\begin{proof}
This result above is \cite[Lemma 3]{dukeexceptional}.
\end{proof}

\begin{corollary}\label{cor:Er}
Let $\ell\neq p$ be an odd prime and let $r\ge 1$. Then one has
\[
\sum_{\substack{|t|\le 2\sqrt{q}\\ t^2\equiv 4q \!\!\!\pmod{\ell^{2r}}}}
H(t^2-4q)
=
\frac{2\ell^{2r}}{\ell^{4r}-1}\,p
+O\!\left(\ell^{2r}\sqrt{p}\right),
\]
and
\[
\sum_{\substack{|t|\le 2\sqrt{q}\\ t^2\equiv 4q \!\!\!\pmod{\ell^{2r+2}}}}
H(t^2-4q)
=
\frac{2\ell^{2r+2}}{\ell^{4r+4}-1}\,p
+O\!\left(\ell^{2r+2}\sqrt{p}\right).
\]
\end{corollary}

\begin{proof}
We apply Theorem~\ref{thm:hurwitz} with $N=\ell^{2r}$ and $N=\ell^{2r+2}$, respectively. The congruence condition
\[
t^2\equiv 4q \pmod{N}
\]
determines exactly two residue classes modulo $N$, corresponding to $t\equiv \pm 2\sqrt{q}\pmod{N}$. Summing the contribution of each class and using the fact that $\ell\nmid p$ yields the stated main terms, while the error terms follow directly from the bound in Theorem~\ref{thm:hurwitz}.
\end{proof}

\begin{theorem}\label{thm:main-density}
With notation as above, one has for $r\geq 1$,
\[
\begin{split}\#\mathcal{E}(r;\F_q)
=&
\frac{2\ell^{2r}}{\ell^{4r}-1}\,p
-
\frac{2\ell^{2r+2}}{\ell^{4r+4}-1}\,p
+O\!\left(\ell^{2r+2}\sqrt{p}\right)\\
=&
\frac{2\ell^{2r}(\ell^2-1)(\ell^{4r+2}+1)}
{(\ell^{4r}-1)(\ell^{4r+4}-1)}\,p
+O\!\left(\ell^{2r+2}\sqrt{p}\right).
\end{split}
\]
On the other hand, 
\[\#\mathcal{E}(r;\F_q)=2\left(1-\frac{\ell^2}{\ell^4-1}\right)p+O\left(\ell^2\sqrt{p}\right).\]
\end{theorem}

\begin{proof}
By definition, $\mathcal{E}(r;\F_q)$ consists precisely of ordinary elliptic curves whose Frobenius discriminant $\Delta=t^2-4q$ satisfies
\[
v_\ell(\Delta)\in\{2r,2r+1\}.
\]
Using the identity $N(t)=H(t^2-4q)$ and subtracting the contributions with valuation at least $2r+2$, we obtain
\[
\#\mathcal{E}(r;\F_q)
=
\sum_{\substack{|t|\le 2\sqrt{q}\\ t^2\equiv 4q \!\!\!\pmod{\ell^{2r}}}}
H(t^2-4q)
-
\sum_{\substack{|t|\le 2\sqrt{q}\\ t^2\equiv 4q \!\!\!\pmod{\ell^{2r+2}}}}
H(t^2-4q).
\]
Substituting the asymptotic formulas from Corollary~\ref{cor:Er} and simplifying yields the stated expressions.
\end{proof}

\begin{theorem}\label{thm:density}
The limiting density
\[
\mathfrak{d}_r
=
\lim_{q\to\infty}\frac{\#\mathcal{E}(r;\F_q)}{\#\mathcal{E}(\F_q)}
\]
exists and is given by
\[
\mathfrak{d}_r
=
\begin{cases}\frac{\ell^{2r}(\ell^2-1)(\ell^{4r+2}+1)}
{(\ell^{4r}-1)(\ell^{4r+4}-1)}&\text{ if }r\geq 1,\\
\left(1-\frac{\ell^2}{\ell^4-1}\right)&\text{ if }r=0.
\end{cases}
\]
\end{theorem}

\begin{proof}
This follows immediately from Theorem~\ref{thm:main-density} together with the asymptotic
$\#\mathcal{E}(\F_q)=2q+O(1)$.
\end{proof}
\bibliographystyle{alpha}
\bibliography{references}

@book {kohel,
    AUTHOR = {Kohel, David Russell},
     TITLE = {Endomorphism rings of elliptic curves over finite fields},
      NOTE = {Thesis (Ph.D.)--University of California, Berkeley},
 PUBLISHER = {ProQuest LLC, Ann Arbor, MI},
      YEAR = {1996},
     PAGES = {117},
      ISBN = {978-0591-32123-4},
   MRCLASS = {Thesis},
}

@article {abvar1,
    AUTHOR = {Brooks, Ernest Hunter and Jetchev, Dimitar and Wesolowski,
              Benjamin},
     TITLE = {Isogeny graphs of ordinary abelian varieties},
   JOURNAL = {Res. Number Theory},
  FJOURNAL = {Research in Number Theory},
    VOLUME = {3},
      YEAR = {2017},
     PAGES = {Paper No. 28, 38},
}

@misc{inverse,
 author = {Bambury, Henry and Campagna, Francesco and Pazuki, Fabien},
 title = {Ordinary isogeny graphs over $\mathbb{F}_p$: the inverse volcano problem},
 year = {2022},
 howpublished = {Preprint, {arXiv}:2210.01086 [math.{NT}] (2022)},
 url = {https://arxiv.org/abs/2210.01086},
 arXiv = {arXiv:2210.01086}
}

@article {abvar2,
    AUTHOR = {Jetchev, Dimitar and Wesolowski, Benjamin},
     TITLE = {Horizontal isogeny graphs of ordinary abelian varieties and
              the discrete logarithm problem},
   JOURNAL = {Acta Arith.},
  FJOURNAL = {Acta Arithmetica},
    VOLUME = {187},
      YEAR = {2019},
    NUMBER = {4},
     PAGES = {381--404},
}

@book {serre1,
    AUTHOR = {Serre, Jean-Pierre},
     TITLE = {Abelian {$l$}-adic representations and elliptic curves},
      NOTE = {McGill University lecture notes written with the collaboration
              of Willem Kuyk and John Labute},
 PUBLISHER = {W. A. Benjamin, Inc., New York-Amsterdam},
      YEAR = {1968},
     PAGES = {xvi+177 pp.},
}

@article {dukeexceptional,
    AUTHOR = {Duke, William},
     TITLE = {Elliptic curves with no exceptional primes},
   JOURNAL = {C. R. Acad. Sci. Paris S\'{e}r. I Math.},
  FJOURNAL = {Comptes Rendus de l'Acad\'{e}mie des Sciences. S\'{e}rie I.
              Math\'{e}matique},
    VOLUME = {325},
      YEAR = {1997},
    NUMBER = {8},
     PAGES = {813--818},
}

@article {CGL,
    AUTHOR = {Charles, Denis X. and Lauter, Kristin E. and Goren, Eyal Z.},
     TITLE = {Cryptographic hash functions from expander graphs},
   JOURNAL = {J. Cryptology},
  FJOURNAL = {Journal of Cryptology. The Journal of the International
              Association for Cryptologic Research},
    VOLUME = {22},
      YEAR = {2009},
    NUMBER = {1},
     PAGES = {93--113},
}

@article {crypto1,
    AUTHOR = {Charles, Denis X. and Lauter, Kristin E. and Goren, Eyal Z.},
     TITLE = {Cryptographic hash functions from expander graphs},
   JOURNAL = {J. Cryptology},
  FJOURNAL = {Journal of Cryptology. The Journal of the International
              Association for Cryptologic Research},
    VOLUME = {22},
      YEAR = {2009},
    NUMBER = {1},
     PAGES = {93--113},
}

@incollection {crypto2,
    AUTHOR = {Feo, Luca De and Fouotsa, Tako Boris and Kutas, P\'{e}ter and
              Leroux, Antonin and Merz, Simon-Philipp and Panny, Lorenz and
              Wesolowski, Benjamin},
     TITLE = {S{CALLOP}: scaling the {CSI}-{F}i{S}h},
 BOOKTITLE = {Public-key cryptography---{PKC} 2023. {P}art {I}},
    SERIES = {Lecture Notes in Comput. Sci.},
    VOLUME = {13940},
     PAGES = {345--375},
}

@article {orvis,
    AUTHOR = {Orvis, Eli},
     TITLE = {Distribution of cycles in supersingular {$\ell$}-isogeny
              graphs},
   JOURNAL = {J. Number Theory},
  FJOURNAL = {Journal of Number Theory},
    VOLUME = {277},
      YEAR = {2025},
     PAGES = {236--261},
      ISSN = {0022-314X,1096-1658},
}

@article {colokohel,
    AUTHOR = {Col\'o, Leonardo and Kohel, David},
     TITLE = {Orienting supersingular isogeny graphs},
   JOURNAL = {J. Math. Cryptol.},
  FJOURNAL = {Journal of Mathematical Cryptology},
    VOLUME = {14},
      YEAR = {2020},
    NUMBER = {1},
     PAGES = {414--437},
      ISSN = {1862-2976,1862-2984},
   MRCLASS = {11G05 (11G07 11G15 11T71 14H10 14H52 14K02 14K22)},
  MRNUMBER = {4165916},
MRREVIEWER = {Sunil\ Chetty},
       DOI = {10.1515/jmc-2019-0034},
       URL = {https://doi.org/10.1515/jmc-2019-0034},
}

@article {cyclesandcuts,
    AUTHOR = {Arpin, Sarah and Bowden, Ross and Clements, James and
              Ghantous, Wissam and LeGrow, Jason T. and Maughan, Krystal},
     TITLE = {Cycles and cuts in supersingular {$L$}-isogeny graphs},
   JOURNAL = {Finite Fields Appl.},
  FJOURNAL = {Finite Fields and their Applications},
    VOLUME = {111},
      YEAR = {2026},
     PAGES = {Paper No. 102768},
}

@incollection {arpinetal,
    AUTHOR = {Arpin, Sarah and Chen, Mingjie and Lauter, Kristin E. and
              Scheidler, Renate and Stange, Katherine E. and Tran, Ha T. N.},
     TITLE = {Orientations and cycles in supersingular isogeny graphs},
 BOOKTITLE = {Research directions in number theory},
    SERIES = {Assoc. Women Math. Ser.},
    VOLUME = {33},
     PAGES = {25--86},
 PUBLISHER = {Springer, Cham},
      YEAR = {[2024] \copyright 2024},
}

@article {duketoth,
    AUTHOR = {Duke, W. and T\'{o}th, \'{A}.},
     TITLE = {The splitting of primes in division fields of elliptic curves},
   JOURNAL = {Experiment. Math.},
  FJOURNAL = {Experimental Mathematics},
    VOLUME = {11},
      YEAR = {2002},
    NUMBER = {4},
     PAGES = {555--565 (2003)},
}

@article{deuring1,
 author = {Deuring, Max},
 title = {Die {Typen} der {Multiplikatorenringe} elliptischer {Funktionenk{\"o}rper}.},
 fjournal = {Abhandlungen aus dem Mathematischen Seminar der Universit{\"a}t Hamburg},
 journal = {Abh. Math. Semin. Univ. Hamb.},
 issn = {0025-5858},
 volume = {14},
 pages = {197--272},
 year = {1941},
}

@article {Schoof,
    AUTHOR = {Schoof, Ren\'{e}},
     TITLE = {Nonsingular plane cubic curves over finite fields},
   JOURNAL = {J. Combin. Theory Ser. A},
  FJOURNAL = {Journal of Combinatorial Theory. Series A},
    VOLUME = {46},
      YEAR = {1987},
    NUMBER = {2},
     PAGES = {183--211},
}

@article{deuring2,
 author = {Deuring, Max},
 title = {Die {Anzahl} der {Typen} von {Maximalordnungen} in einer {Quaternionenalgebra} von primer {Grundzahl}},
 fjournal = {Nachrichten der Akademie der Wissenschaften in G{\"o}ttingen, Mathematisch-Physikalische Klasse. 2a, Mathematisch-Physikalisch-Chemische Abteilung},
 journal = {Nachr. Akad. Wiss. G{\"o}ttingen, Math.-Phys. Kl., Math.-Phys.-Chem. Abt.},
 issn = {0369-3163},
 volume = {1945},
 pages = {48--50},
 year = {1945},
}

@article{waterhouse,
 author = {Waterhouse, W. C.},
 title = {Abelian varieties over finite fields},
 fjournal = {Annales Scientifiques de l'{\'E}cole Normale Sup{\'e}rieure. Quatri{\`e}me S{\'e}rie},
 journal = {Ann. Sci. {\'E}c. Norm. Sup{\'e}r. (4)},
 issn = {0012-9593},
 volume = {2},
 pages = {521--560},
 year = {1969},
}

@article {jonesalmostall,
    AUTHOR = {Jones, Nathan},
     TITLE = {Almost all elliptic curves are {S}erre curves},
   JOURNAL = {Trans. Amer. Math. Soc.},
  FJOURNAL = {Transactions of the American Mathematical Society},
    VOLUME = {362},
      YEAR = {2010},
    NUMBER = {3},
     PAGES = {1547--1570},
}

@article {serre2,
    AUTHOR = {Serre, Jean-Pierre},
     TITLE = {Propri\'{e}t\'{e}s galoisiennes des points d'ordre fini des
              courbes elliptiques},
   JOURNAL = {Invent. Math.},
  FJOURNAL = {Inventiones Mathematicae},
    VOLUME = {15},
      YEAR = {1972},
    NUMBER = {4},
     PAGES = {259--331},
}

@incollection{sutherlandvolcanoes,
 author = {Sutherland, Andrew V.},
 title = {Isogeny volcanoes},
 booktitle = {ANTS X. Proceedings of the tenth algorithmic number theory symposium, San Diego, CA, USA, July 9--13, 2012},
 isbn = {978-1-935107-00-2; 978-1-935107-01-9},
 pages = {507--530},
 year = {2013},
 publisher = {Berkeley, CA: Mathematical Sciences Publishers (MSP)},
}

\end{document}